\documentclass[letterpaper,11pt]{article}
\usepackage{amsmath,amssymb,amsthm}
\usepackage[margin=1.25in]{geometry}

\newtheorem{thm}{Theorem}[section]
\newtheorem{corollary}[thm]{Corollary}
\newtheorem{lemma}[thm]{Lemma}
\newtheorem{prop}[thm]{Proposition}

\numberwithin{equation}{section}

\def\R{{\mathbb R}}
\def\P{{\mathbb P}}
\def\E{{\mathbb E}}
\def\1{{\bf 1}}

\renewcommand{\Pr}{\mathbb P}
\newcommand{\ip}[1]{\langle #1 \rangle}

\renewcommand{\bar}{\overline}

\begin{document}

\title{\Large \bf
  Global Heat Kernel Estimates for Fractional Laplacians in Unbounded
  Open Sets
  \thanks{Research partially supported by NSF Grant
 DMS-0600206.}
}

\author{{\bf Zhen-Qing Chen} \quad \hbox{and} \quad {\bf Joshua Tokle}}

\date{(June 5, 2009)}
\maketitle

\begin{abstract}
In this paper, we  derive global sharp heat kernel estimates for
symmetric $\alpha$-stable processes (or equivalently, for the
fractional Laplacian with zero exterior condition) in two classes of
unbounded $C^{1,1}$ open sets in $\R^d$: half-space-like open sets
and exterior open sets. These open sets can be disconnected. We
focus in particular on explicit estimates for $p_D(t,x,y)$ for all
$t>0$ and $x, y\in D$. Our approach is based on the idea that for
$x$ and $y$ in $D$ far from the boundary and $t$ sufficiently large,
we can compare $p_D(t,x,y)$ to the heat kernel in a well understood
open set: either a half-space or $\R^d$;
 while for the general case we can reduce them
to the above case by pushing $x$ and $y$ inside away from the boundary.
As a consequence, sharp Green functions estimates are obtained for
the Dirichlet fractional Laplacian in these two types of
open sets.
Global sharp heat kernel estimates and Green function estimates
are also obtained for censored stable processes (or equivalently,
for regional fractional Laplacian) in exterior open sets.
\end{abstract}



\bigskip

\section{Introduction}
Fix an integer $d \ge 1$ and a real number $\alpha \in (0,2)$.   A $d$-dimensional symmetric $\alpha$-stable process is a L\'evy process $X = \{ X_t, t \ge 0, \Pr_x, x \in \R^d \}$ such that
\begin{align*}
    \E_x\left[ e^{i\ip{\xi, X_t - X_0}} \right] = e^{-t|\xi|}
    \qquad
    \text{for all $x \in \R^d$ and $\xi \in \R^d$.}
\end{align*}
Here $\ip{x, y}$ denotes the inner product for $x, y \in \R^d$.
The infinitesimal generator of a symmetric $\alpha$-stable process $X$ in $\R^d$ is the fraction Laplacian $-(-\Delta)^{\alpha/2}$, a prototype of nonlocal operators.  The fractional Laplacian can be written in the form
\begin{align*}
    -(-\Delta)^{\alpha/2}u(x) = c \lim_{\varepsilon \downarrow 0}
    \int_{\{y \in \R^d: |y-x| > \varepsilon\}} \left( u(y) - u(x) \right)\frac{dy}{|x-y|^{d+\alpha}}
\end{align*}
for some specific constant $c = c(d,\alpha)$ that depends only on $d$ and $\alpha$..

A symmetric $\alpha$-stable process $X$ has a H\"older continuous transition density $p(t,x,y)$ defined on $(0,\infty)\times \R \times \R$, also known in literature as the heat kernel for the fractional Laplacian.  It is well known (see, e.g., \cite{CK}) that there exists a constant $C_1 = C_1(d,\alpha) > 0$ such that
\begin{align}
    C_1^{-1} \left( t^{-d/\alpha} \wedge \frac{t}{|x-y|^{d+\alpha}} \right)
    \le p(t,x,y)
    \le C_1 \left( t^{-d/\alpha} \wedge \frac{t}{|x-y|^{d+\alpha}} \right)
    \label{E:global-bound}
\end{align}
for all $x,y \in \R$ and $t>0$.  Here and in the sequel, $|x-y|$ denotes the  Euclidean distance between $x$ and $y$, and for two real numbers $a$ and $b$ we set $a\wedge b := \min\{a,b\}$ and $a\vee b := \max\{a,b\}$.  For two nonnegative functions $f$ and $g$ the notation $f\asymp g$ means that there are positive constants $c_1$ and $c_2$ such that $c_1f(x) \le g(x) \le c_2f(x)$ for all $x$ in the common domain of $f$ and $g$.

For an open subset $D \subset \mathbb R^d$, we define $X^D$ to be the subprocess of $X$ killed upon  exiting $D$.  That is, if we write $\tau_D := \inf\left\{ t > 0 : X_t \not\in D \right\}$ for the exit time of $X$ from $D$, then
$X^D_t(\omega) = X_t(\omega)$ when $t<\tau_D (\omega)$ and $X^D_t(\omega)= \partial $ when $t\geq \tau_D(\omega)$. Here $\partial $ is a cemetery point added to $\R^d$.
  The infinitesimal generator of $X^D$ is the Dirichlet fractional Laplacian $-(-\Delta)^{\alpha/2}\vert_D$ (the fractional Laplacian with zero exterior condition).  The transition density of $X^D$ (equivalently, the heat kernel for the Dirichlet fractional Laplacian)  will be denoted by $p_D(t,x,y)$. By Dynkin's formula, one has for $x, y\in D$ that
\begin{align}\label{E:dynkin's-formula}
    p_D(t,x,y) = p(t,x,y) - \E_x\left[ p\left( t-\tau_D, X_{\tau_D}, y \right);
    \tau_D < t \right].
\end{align}

Recall that an open subset $D\subset \R^d$ is said to be a $C^{1,1}$ open set if there exist a localization radius $R_0 > 0$ and a constant $\Lambda_0 > 0$ such that for every $z \in \partial D$ there is a $C^{1,1}$ function $\phi = \phi_z : \mathbb R^{d-1} \to \mathbb R$ satisfying $\phi(0) = \nabla \phi(0) = 0$, $\|\nabla \phi \|_\infty \le \Lambda_0$, $|\nabla \phi(x) - \nabla \phi(z)| \le \Lambda_0|x-z|$, and an orthonormal coordinate system $y = (y_1, \dots, y_{d-1},y_d) := (\tilde{y},y_d)$ such that $B(z,R_0)\cap D = B(z,R_0)\cap \left\{ y: y_d > \phi(\tilde{y}) \right\}$.  Together $R_0$ and $\Lambda_0$ are called the $C^{1,1}$ characteristics of $D$.

The following was recently established in \cite{CKS}.

\begin{thm}[Theorem 1.1 of \cite{CKS}]\label{T:1.1}
 Let $D$ be a $C^{1,1}$ open
subset of $\R^d$ with $d\geq 1$ and $\delta_D(x)$ the Euclidean
distance between $x$ and $D^c$.
\begin{description}
\item{\rm (i)} For every $T>0$, on $(0, T]\times D\times D$,
$$
p_D(t, x, y)\,\asymp\, \left( 1\wedge
\frac{\delta_D(x)}{t^{1/\alpha}}\right)^{\alpha/2} \left( 1\wedge
\frac{\delta_D(y)}{t^{1/\alpha}}\right)^{\alpha/2} \left( t^{-d/\alpha}
\wedge \frac{t}{|x-y|^{d+\alpha}}\right) .
$$
\item{\rm (ii)} Suppose in addition that $D$ is bounded.
For every $T>0$, there are positive constants $c_1<c_2$ so that on
$[T, \infty)\times D\times D$,
$$
c_1\, e^{-\lambda_1 t}\, \delta_D (x)^{\alpha/2}\, \delta_D
(y)^{\alpha/2} \,\leq\,  p_D(t, x, y) \,\leq\, c_2\, e^{-\lambda_1
t}\, \delta_D (x)^{\alpha/2} \,\delta_D (y)^{\alpha/2} ,
$$
where $\lambda_1>0$ is the smallest eigenvalue of the Dirichlet
fractional Laplacian $(-\Delta)^{\alpha/2}|_D$.
\end{description}
\end{thm}

In this paper, we further the study of heat kernel estimates for symmetric $\alpha$-stable
processes in unbounded $C^{1,1}$ open sets.
We will be concerned with global sharp two-sided estimates for $p_D(t,x,y)$ for two families of unbounded $C^{1,1}$ open sets $D$.
We focus in particular on explicit estimates for $p_D(t,x,y)$ for all
$t>0$ and $x, y \in D$ in terms of distance functions.
Define a half-space to be any isometry of the usual upper half-space
$\left\{ (x_1,\dots,x_d): x_d > 0 \right\}$.  We say that $D$ is
half-space-like if, after isometry,  $H_a \subset D \subset H_b$ for
some real numbers $a>b$. Here, for a real number $a$, $H_a:=\left\{
(x_1,\dots,x_d)\in \R^d: x_d > a \right\}$.
  Note that domains lying above the graph of bounded $C^{1,1}$ functions
 are half-space-like $C^{1,1}$ connected open sets.
 We say that $D$ is
an exterior open set if $D^c$ is compact.
 Our main results are the following.

\begin{thm}\label{T:half-space-like}
    Suppose $D$ is a half-space-like $C^{1,1}$ open set in $\R^d$ with $d\geq 1$.  Then
    for $\alpha \in (0, 2)$, on $\R_+\times D \times D$
    \begin{align}
        p_D(t,x,y) \asymp
        \left( 1\wedge \frac{\delta_D(x) }{t^{1/\alpha}} \right)^{\alpha/2}
        \left( 1\wedge \frac{\delta_D(y)}{t^{1/\alpha}} \right)^{\alpha/2}
        \left( t^{-d/\alpha}\wedge \frac{t}{|x-y|^{d+\alpha}} \right).
        \label{E:half-space-like}
    \end{align}
\end{thm}

\begin{thm}\label{T:exterior}
    Suppose $D$ is an exterior open $C^{1,1}$ set in $\R^d$ with  $d \geq 1$ and
    $d>\alpha$.  Then on $\R_+\times D \times D$
    \begin{align}
        p_D(t,x,y) \asymp
        \left( 1\wedge \frac{\delta_D(x)}{1\wedge t^{1/\alpha}} \right)^{\alpha/2}
        \left( 1\wedge \frac{\delta_D(y)}{1\wedge t^{1/\alpha}} \right)^{\alpha/2}
        \left( t^{-d/\alpha}\wedge \frac{t}{|x-y|^{d+\alpha}} \right).
        \label{E:exterior}
    \end{align}
\end{thm}

It's worth taking a moment to look at the forms of these estimates.
The estimate in \eqref{E:half-space-like} is the same estimate that
appears in Theorem \ref{T:1.1}(i).  This suggests that for a
half-space-like open set, the large time behavior is not
qualitatively different from the small time behavior.  Contrast this
with \eqref{E:exterior}, the estimate for an exterior open set,
where for large times $t$ the boundary terms do not depend on $t$.
The difference reflects the fact that the symmetric $\alpha$-stable process
$X$ will hit $D^c$ with probability one when $D$ is a
half-space-like open set; with positive probability, however, $X$ wanders
to infinity without hitting $D^c$ when $D$ is an exterior open set
in $\R^d$ with $d>\alpha$.

The Green function of $X$ in $D$ is given by $G_D(x,y) = \int_0^\infty p_D(t,x,y)dt$.  By integrating the estimates that appear in Theorems \ref{T:half-space-like} and \ref{T:exterior} with respect to time, we obtain the following estimates for $G_D(x,y)$.

\begin{corollary}\label{C:half-space-like}
    Suppose $D$ is a half-space-like $C^{1,1}$ open set in $\R^d$ with $d\geq 1$.
        \begin{align*}
            G_D(x,y) \asymp \begin{cases}
                \frac{1}{|x-y|^{d-\alpha}}
                \left( 1\wedge \frac{\delta_D(x)\delta_D(y)}{|x-y|^2} \right)^{\alpha/2} & \text{when $d>\alpha$,} \\
                \log\left( 1 + \frac{\delta_D(x)^{\alpha/2}\delta_D(y)^{\alpha/2}}{|x-y|^\alpha} \right) & \text{when $d=1=\alpha$,} \\
                \left( \delta_D(x)\delta_D(y) \right)^{(\alpha-1)/2}
                \wedge \frac{\delta_D(x)^{\alpha/2}\delta_D(y)^{\alpha/2}}{|x-y|}
                & \text{when $d=1<\alpha$.}
            \end{cases}
        \end{align*}
\end{corollary}

\begin{corollary}\label{C:exterior}
    Suppose $D$ is an exterior $C^{1,1}$ open set in $\R^d$ with  $d \geq 1$ and $d>\alpha$.  Then
    \begin{align*}
        G_D(x,y) \asymp \frac{1}{|x-y|^{d-\alpha}}
        \left( 1\wedge \frac{\delta_D(x)^{\alpha/2}}{1\wedge |x-y|^{\alpha/2} } \right)
        \left( 1\wedge \frac{\delta_D(y)^{\alpha/2} }{1\wedge |x-y|^{\alpha/2} } \right) .
    \end{align*}
\end{corollary}

The proofs of Theorems \ref{T:half-space-like} and \ref{T:exterior}
will be given in  Sections \ref{sec:half-space-like} and \ref{sec:exterior},
respectively.  Both proofs rely on the idea that for $x$ and $y$ in
$D$ far from the boundary and $t$ sufficiently large, we can compare
$p_D(t,x,y)$ to the heat kernel in a well understood open set: either a
half-space or $\mathbb R^d$.  Once this is established, we extend
our estimate to the boundary of $D$ using the following corollary to
Theorem \ref{T:1.1}.  The result will allow us to ``push'' points in
$D$ a fixed distance away from $\partial D$.

\begin{lemma}\label{L:ratio}
    Let $D$ be a $C^{1,1}$ open set, and let $\lambda >0$, $t_0 >0$ be fixed.  Suppose $x,x_0 \in D$ satisfy $|x-x_0| = \lambda t_0^{1/\alpha}$.  Then
    \begin{align}
        \frac{p_D(t_0,x,z)}{p_D(t_0,x_0,z)} \asymp
        \frac{ 1\wedge \delta_D(x)^{\alpha/2}}{1\wedge
        \delta_D(x_0)^{\alpha/2}},
        \label{E:ratio}
    \end{align}
    where the (implicit) comparison constants
     in  \eqref{E:ratio}  depend only on $d$, $\alpha$, $\lambda$, $t_0$
     and the $C^{1,1}$ characteristics of $D$.
\end{lemma}

\begin{proof}
    By symmetry, it suffices to prove an upper bound.  By Theorem \ref{T:1.1} there exists a $c_1>0$ such that
    \begin{align*}
        p_D(t_0,x,z) \le c_1
        \left( 1\wedge \frac{\delta_D(x) }{t_0^{1/\alpha}} \right)^{\alpha/2}
        \left( 1\wedge \frac{\delta_D(z)}{t_0^{1/\alpha}} \right)^{\alpha/2}
        t_0^{-d/\alpha}
        \left( 1\wedge \frac{t_0^{1/\alpha}}{|x-z|} \right)^{d+\alpha}
    \end{align*}
    and
    \begin{align*}
        p_D(t_0,x_0,z) \ge c_1^{-1}
        \left( 1\wedge \frac{\delta_D(x_0) }{t_0^{1/\alpha}} \right)^{\alpha/2}
        \left( 1\wedge \frac{\delta_D(z)}{t_0^{1/\alpha}} \right)^{\alpha/2}
        t_0^{-d/\alpha}
        \left( 1\wedge \frac{t_0^{1/\alpha}}{|x_0-z|} \right)^{d+\alpha}.
    \end{align*}
    Thus
    \begin{align*}
        \frac{p_D(t_0,x,z)}{p_D(t_0,x_0,z)} \le c_1^2 \,
        \frac{  1\wedge \delta_D(x)^{\alpha /2}}{  1\wedge \delta_D(x_0)^{\alpha/2}}
        \left( \frac{1\wedge \frac{t_0^{1/\alpha}}{|x-z|}}{1\wedge \frac{t_0^{1/\alpha}}{|x_0-z|}} \right)^{d+\alpha}
    \end{align*}
    To complete the proof, we must show that the last term on the right-hand side is bounded by a constant.  Define $\lambda_0 := \lambda/4$.  Then for $z \in B(x_0,\lambda_0 t_0^{1/\alpha})$ we have $|x-z| \asymp t_0^{1/\alpha}$, and similarly, for $z \in B(x,\lambda_0 t_0^{1/\alpha})$ we have $|x_0-z| \asymp t_0^{1/\alpha}$.  Finally, for $z \not\in B(x,\lambda_0 t_0^{1/\alpha})\cup B(x_0,\lambda_0 t_0^{1/\alpha})$ we have $|x-z| \asymp |x_0-z|$.
\end{proof}

In recent years, progress has been made proving these kinds of sharp
two-sided heat kernel estimates for Dirichlet Laplacians, that is,
for the transition densities of killed Brownian motions.  The first
such estimate appears in \cite{Z1},  wherein Zhang proves the
analogy of Theorem \ref{T:1.1}.  In \cite{Z2}, Zhang proves a global
estimate in the case where $D$ is an exterior $C^{1,1}$ connected
open set; his proof uses the interior estimates established by
Grigor\'yan and Saloff-Coste in \cite{GSC}.  Then, in \cite{S}, Song
proves a global estimate in the case where $D$ is given by the
region above a bounded $C^{1,1}$ function.
See \cite{GyS} for  recent development on heat kernel estimates for
Dirichlet Laplacian in inner uniform domains.

 The estimates in Theorems \ref{T:half-space-like} and
\ref{T:exterior} are the analogs in the symmetric $\alpha$-stable
process case of the sharp heat kernel estimates given in \cite{Z2}
and \cite{S}.  Our technique, based on comparison,  differs from the
one used in those papers, but it can be adapted to give new proofs
in the Brownian motion case.  Moreover, the proofs in $\cite{Z1}$,
\cite{Z2}, and \cite{S} rely on a sophisticated backward parabolic
Harnack inequality (the local comparison theorem of \cite{FGS}), a
result whose analogy for stable processes and non-local operators is
not known and seems difficult to establish.  Our technique,
therefore, is more elementary and, we believe, more perspicuous.

The approach and ideas developed in this paper can be used
to study heat kernel estimates for other types of discontinuous
processes in unbounded open sets.
 For instance, in
Section \ref{sec:c-exterior},
 we will establish
an estimate for the heat kernel of the censored $\alpha$-stable
process in an exterior domain.  We plan to address the sharp heat
kernel estimates for mixed stable processes in unbounded
 open sets
in a separate paper.

For an open subset $D$ of $\R^d$, a censored $\alpha$-stable process $Y$ in $D$ is a strong Markov process whose infinitesimal generator is given by
\begin{align*}
    \mathcal L_Du(x) := c\,
    \lim_{\varepsilon\downarrow 0}\int_{\{y\in D: |y-x|>\varepsilon\}}
    (u(y) - u(x))\frac{dy}{|x-y|^{d+\alpha}},
\end{align*}
for a constant $c = c(d,\alpha)$.
In some literature, $\mathcal L_D$ is called the regional fractional
Laplacian.
The operator $\mathcal L_D$ differs from the fractional Laplacian in that the integration in the definition is taken over $y$ in $D$ only.  We will write $q_D(t,x,y)$ for the heat kernel of $Y$.  Such processes were studied extensively in \cite{BBC} where it is shown, for example, that when $D$ is a Lipschitz domain, then the process $Y$ coincides with the reflected $\alpha$-stable process $\bar{Y}$ in $D$ if and only if $\alpha \leq 1$.
Two-sided heat kernel estimates for $\bar Y$ have been studied in \cite{CK}.

For $\alpha \in (1,2)$,   $Y$ is the subprocess of $\bar{Y}$ killed upon exiting $D$, and if $D$ is a bounded Lipschitz open set then $Y$ exits $D$ continuously and in finite time (see \cite[Theorem 1.1]{BBC}).  That is, if we let $\zeta := \inf\left\{ t>0: Y_t \not\in D \right\}$ denote the lifetime of $Y$, then $\Pr_x(\zeta < \infty)=1$ for all $x \in D$ and $Y_{\zeta-} \in \partial D$.  In \cite{CKS2}, the following analog to Theorem \ref{T:1.1} was given for censored $\alpha$-stable processes.

\begin{thm}[Theorem 1.1 of \cite{CKS2}]\label{T:c-1.1}
 Let $D$ be a $C^{1,1}$ open
subset of $\R^d$ with $d\geq 1$ and $\delta_D(x)$ the Euclidean
distance between $x$ and $D^c$.  Suppose $\alpha \in (1,2)$.
\begin{description}
\item{\rm (i)} For every $T>0$, on $(0, T]\times D\times D$,
$$
q_D(t, x, y)\,\asymp\,
\left( 1\wedge \frac{\delta_D(x)}{t^{1/\alpha}}\right)^{\alpha-1}
\left( 1\wedge \frac{\delta_D(y)}{t^{1/\alpha}}\right)^{\alpha-1}
\left( t^{-d/\alpha} \wedge \frac{t}{|x-y|^{d+\alpha}}\right) .
$$
\item{\rm (ii)} Suppose in addition that $D$ is bounded.
For every $T>0$, there are positive constants $c_1<c_2$ so that on
$[T, \infty)\times D\times D$,
$$
c_1\, e^{-\lambda_1 t}\, \delta_D (x)^{\alpha-1}\, \delta_D
(y)^{\alpha-1} \,\leq\,  q_D(t, x, y) \,\leq\, c_2\, e^{-\lambda_1
t}\, \delta_D (x)^{\alpha-1} \,\delta_D (y)^{\alpha-1} ,
$$
where $\lambda_1>0$ is the smallest eigenvalue of $-\mathcal L_D$.
\end{description}

\end{thm}
A main difficulty in studying $\mathcal L_D$ versus
$\Delta^{\alpha/2}\vert_D$  is that we no longer have domain
monotonicity.  That is, for $D_1 \subset D_2$, it does not necessarily
follow that $q_{D_1}(t,x,y) \le q_{D_2}(t,x,y)$.  For this reason,
our proof of Theorem \ref{T:half-space-like} does not extend
naturally to censored stable processes.
However, we are able to adapt our approach to Theorem \ref{T:exterior}
to obtain the following sharp heat kernel estimates for censored stable processes
in exterior
 open sets;
 the proof is outlined in section \ref{sec:c-exterior}.

\begin{thm}\label{T:c-exterior}
    Suppose $\alpha \in (1,2)$, $D$ is an exterior open $C^{1,1}$ set in $\R^d$, and $d \geq 2$.  Then
    \begin{align}
        q_D(t,x,y) \asymp
        \left( 1\wedge \frac{\delta_D(x)}{1\wedge t^{1/\alpha}} \right)^{\alpha-1}
        \left( 1\wedge \frac{\delta_D(y)}{1\wedge t^{1/\alpha}} \right)^{\alpha-1}
        \left( t^{-d/\alpha}\wedge \frac{t}{|x-y|^{d+\alpha}} \right).
        \label{E:c-exterior}
    \end{align}
\end{thm}
Once again, we can integrate the heat kernel estimate to get a Green function estimate.
\begin{corollary}\label{C:c-exterior}
    Suppose $\alpha \in (1,2)$ and $D$ is an exterior open $C^{1,1}$ set in $\R^d$ with $d \geq 2 > \alpha$.
    \begin{align*}
        G_D(x,y) \asymp \frac{1}{|x-y|^{d-\alpha}}
        \left( 1\wedge \frac{\delta_D(x)^{\alpha-1}}{1\wedge |x-y|^{\alpha/2}} \right)
        \left( 1\wedge \frac{\delta_D(y)^{\alpha-1}}{1\wedge |x-y|^{\alpha/2}} \right).
    \end{align*}
\end{corollary}

\medskip

In Theorems \ref{T:exterior} and \ref{T:c-exterior},
it is assumed that $d>\alpha$.
It is natural to ask about global heat kernel
estimates on $p_D(t, x, y)$ and $q_D(t, x, y)$
for exterior open sets in $\R$ with $\alpha
\geq 1$. Obtaining these estimates will new require insights; we plan to address the problem elsewhere.

When we had nearly finished this paper, we saw a preprint \cite{BGR}
by Bogdan, Grzywny, and Ryznar at arXiv.  In that paper the authors
independently obtained the heat kernel estimates of Theorem
\ref{T:exterior} for symmetric stable processes in exterior
$C^{1,1}$ open sets by a different approach.

Throughout this paper, we use $c$, $c_1$, $c_2$, \dots to denote
generic constants,  whose exact values are not important and can
change from one appearance to another.  The labeling of the
constants $c_1$, $c_2$, \dots starts anew in each proof.  The values
of the constants $C_0$, $C_1$, \dots ,  will remain the same
throughout this paper.  $B(x,r)$ denotes the Euclidean ball centered at $x$ with radius $r$.  For a Borel set $E$, we write $\tau_E := \inf\left\{ t>0: X_t\not\in E \right\}$ for the exit time from $E$, $T_E := \inf\left\{ t>0: X_t \in E \right\}$ for the hitting time of $E$, and $|E|$ for the Lebesgue measure of $E$.

\section{Heat kernel estimates on half-space-like open sets}
\label{sec:half-space-like}

In this section, we give a proof of Theorem \ref{T:half-space-like}.
Recall that we define a half-space to be any isometry of the upper
half-space $\left\{ (x_1,\dots,x_d) \in \mathbb R^d: x_d > 0
\right\}$, and we say that $D$ is a half-space-like open set if,
after isometry, $H_a\subset D\subset H_b$ for some $a>b$,
 where $H_a:=\{(x_1, \cdots, x_d)\in R^d: x_d >a\}$.
An elementary but key observation is that for a half-space-like open
set
\begin{align}\label{E:half-space-compare}
    p_{H_a}(t,x,y) \le p_D(t,x,y) \le p_{H_b}(t,x,y).
\end{align}
Moreover, it's easy to see that the desired estimate \eqref{E:half-space-like} holds in the case of a true half-space.

\begin{lemma}
    Let $H$ be a half-space.  Then $p_{H}(t,x,y)$ satisfies \eqref{E:half-space-like} on $(0,\infty)\times H\times H$.
    \label{L:half-space}
\end{lemma}

\begin{proof}
    Without loss of generality, $H$ is the upper half-space $\{(x_1,\dots,x_d): x_d > 0\}$,
    whence by stable scaling $p_{H}(t,x,y) = t^{-d/\alpha}p_{H}(1,t^{-1/\alpha}x,t^{-1/\alpha}y)$.  The result then follows by applying Theorem \ref{T:1.1} with $T=1$.
\end{proof}

We may combine this lemma with \eqref{E:half-space-compare} to obtain a
useful estimate for  $x$ and $y$ well inside $H_a$; the plan, then, is to
reduce our problem to this case.  Without loss of generality, we
assume that $a>0$,  $b=0$ (so $H_b$ is the upper half-space), and
$H_a\subset D\subset H_0$. Define
\begin{align}
    t_0 := (1\vee a)^\alpha.
    \label{eq:t_0}
\end{align}
 Note that  $e_d$ is the inward unit normal at 0
(where $\{ e_1, \dots, e_d\}$ is the implicit orthonormal basis for
$\mathbb R^d$).  For $x$ and $y$ in $D$, define the points
\begin{align*}
    x_0 := x + 2t_0^{1/\alpha}e_d
    \qquad\text{and}\qquad
    y_0 := y + 2t_0^{1/\alpha}e_d.
\end{align*}
Observe that $\delta_D(x_0) \ge \delta_{H_a}(x_0) > t_0^{1/\alpha}$.
Hence, we may sharpen the conclusion of Lemma \ref{L:ratio} as follows:
\begin{align}
    \frac{p_D(t_0,x,z)}{p_D(t_0,x_0,z)} \asymp
    \left( 1\wedge \delta_D(x)^\alpha \right)^{1/2}
    \quad\text{and}\quad
    \frac{p_D(t_0,y,z)}{p_D(t_0,y_0,z)} \asymp
    \left( 1\wedge \delta_D(y)^\alpha \right)^{1/2}.
    \label{E:ratio-half-space-like}
\end{align}

We now proceed to the proof of Theorem \ref{T:half-space-like}.  For clarity's sake, we will state a technical lemma and use it to prove the theorem.  The proof of the lemma will follow.

\begin{lemma}\label{L:inward-push-comparison}
    Assume $t_0 \ge 1$.  There exists a constant $c>0$ such that
    \begin{align*}
        c^{-1}\left( 1\wedge \delta_D(x)^\alpha \right)\left( 1\wedge \frac{\delta_{H_0}(x_0)^\alpha}{t} \right)
        \le 1\wedge \frac{\delta_D(x)^\alpha}{t} \le
        c\left( 1\wedge \delta_D(x)^\alpha \right)\left( 1\wedge \frac{\delta_{H_a}(x_0)^\alpha}{t} \right)
    \end{align*}
    for all $t > 5^\alpha t_0$.
\end{lemma}

\begin{proof}[Proof of Theorem \ref{T:half-space-like}]
    By Theorem \ref{T:1.1}, it suffices to show that there exists a $T>0$ such that \eqref{E:half-space-like} holds on $(T,\infty)\times D\times D$.  Take $T := \max\{2t_0, 5^\alpha t_0\}$ and assume $t > T$.  By the semigroup property and \eqref{E:ratio-half-space-like} we have
    \begin{align*}
        p_D(t,x,y)
        &= \int_D \int_D p_D(t_0,x,z)p_D(t-2t_0,z,w)p_D(t_0,w,y)dzdw \notag\\
        &\asymp \left( 1\wedge \delta_D(x)^\alpha \right)^{1/2}\left( 1\wedge \delta_D(y)^\alpha \right)^{1/2} \\
        &\qquad\qquad \cdot\int_D \int_D p_D(t_0,x_0,z)p_D(t-2t_0,z,w)p_D(t_0,w,y_0)dzdw \\
        &= \left( 1\wedge \delta_D(x)^\alpha \right)^{1/2}\left( 1\wedge \delta_D(y)^\alpha \right)^{1/2}p_D(t,x_0,y_0).
    \end{align*}
    The desired result now follows from \eqref{E:half-space-compare}, Lemma \ref{L:half-space}, and Lemma \ref{L:inward-push-comparison}.
\end{proof}

\begin{proof}[Proof of Lemma \ref{L:inward-push-comparison}]
    For convenience, note that
    \begin{align*}
    \delta_D(x) +   t_0^{1/\alpha} \le \delta_{H_a}(x_0) \le \delta_D(x) + 2t_0^{1/\alpha}
        \quad\text{and}\quad
        \delta_D(x)+2t_0^{1/\alpha} \le \delta_{H_0}(x_0) \le \delta_D(x) + 3t_0^{1/\alpha}.
           \end{align*}

    \bigskip
    \emph{Case 1:} $\delta_D(x) \le 2t_0^{1/\alpha}$.
    Note that $2t_0^{1/\alpha}\geq 2$. We have from the above display that
    \begin{align*}
        \left( 1\wedge \delta_D(x)^\alpha \right)\left( 1\wedge \frac{\delta_{H_a}(x_0)^\alpha}{t} \right)
        \ge \left( \frac{\delta_D(x)^\alpha}{2^\alpha t_0} \right) \left( \frac{t_0}{t} \right)
        = 2^{-\alpha}\frac{\delta_D(x)^\alpha}{t}
    \end{align*}
    and
    \begin{align*}
        \left( 1\wedge \delta_D(x)^\alpha \right)\left( 1\wedge \frac{\delta_{H_0}(x_0)^\alpha}{t} \right)
        \le \delta_D(x)^\alpha \left( \frac{5^\alpha t_0}{t} \right)
        = 5^\alpha t_0\frac{\delta_D(x)^\alpha}{t}.
    \end{align*}

    \bigskip
    \emph{Case 2:} $\delta_D(x) > 2t_0^{1/\alpha}$.  In this case $1\wedge \delta_D(x)^\alpha = 1$ and
    \begin{align*}
        \delta_D(x) \le \delta_{H_a}(x_0) < \delta_{H_0}(x_0) \le \frac{5}{2}\delta_D(x).
    \end{align*}
    The result follows immediately.
\end{proof}

Now that Theorem \ref{T:half-space-like} is established, we give a proof of Corollary \ref{C:half-space-like}, the Green function estimate.  We follow the proof of \cite[Corollary 1.2]{CKS} closely, making some adjustments to account for the fact that the diameter of a half-space-like open set is infinite.  Here and in the proof of Corollary \ref{C:exterior}, we will use the following change of variables computation.  For $T > 0$ we set $u = |x-y|^\alpha/t$ to get
\begin{align}\label{E:Green-u}
    &\int_0^T\left( 1\wedge \frac{\delta_D(x)^{\alpha/2}}{\sqrt t} \right)
    \left( 1\wedge \frac{\delta_D(y)^{\alpha/2}}{\sqrt t} \right)
    \left( t^{-d/\alpha}\wedge \frac{t}{|x-y|^{d+\alpha}} \right)dt \notag \\
    =\quad &\frac{1}{|x-y|^{d-\alpha}}\int_{\frac{|x-y|^\alpha}{T}}^\infty
    \left( 1\wedge \frac{\sqrt u \delta_D(x)^{\alpha/2}}{|x-y|^{\alpha/2}} \right)
    \left( 1\wedge \frac{\sqrt u \delta_D(y)^{\alpha/2}}{|x-y|^{\alpha/2}} \right)
    \left( u^{(d/\alpha)-2}\wedge u^{-3} \right) du.
\end{align}
In addition, the following computation is shown in \cite{CKS}.
\begin{align}\label{E:Green-u-large}
    &\frac{1}{|x-y|^{d-\alpha}}\int_1^\infty
    \left( 1\wedge \frac{\sqrt u \delta_D(x)^{\alpha/2}}{|x-y|^{\alpha/2}} \right)
    \left( 1\wedge \frac{\sqrt u \delta_D(y)^{\alpha/2}}{|x-y|^{\alpha/2}} \right)
    \left( u^{(d/\alpha)-2}\wedge u^{-3} \right)  du \notag \\
    \asymp\quad
    &\frac{1}{|x-y|^{d-\alpha}}
    \left( 1\wedge \frac{\delta_D(x)^{\alpha/2}}{|x-y|^{\alpha/2}} \right)
    \left( 1\wedge \frac{\delta_D(y)^{\alpha/2}}{|x-y|^{\alpha/2}} \right).
\end{align}
Since $\delta_D(x)\leq \delta_D(y)+|x-y|$, it is easy to verify
that
\begin{equation}\label{e:2.7}
\left(1\wedge \frac{\delta_D(x)}{|x-y|}\right)
\left(1\wedge \frac{\delta_D(y)}{|x-y|}\right)
\leq  1\wedge \frac{\delta_D(x)\delta_D(y)}{|x-y|^2}
\leq  2 \left(1\wedge \frac{\delta_D(x)}{|x-y|}\right)
\left(1\wedge \frac{\delta_D(y)}{|x-y|}\right).
\end{equation}

\bigskip

\begin{proof}[Proof of Corollary \ref{C:half-space-like}]
    Assume without loss of generality that $\delta_D(x) \le \delta_D(y)$ and define $T = T(x,y) := \max\{ \delta_D(y)^\alpha, |x-y|^\alpha \}$.  By Theorem \ref{T:half-space-like},
    \begin{align}\label{E:Green-t-large}
        \int_T^\infty p_D(t,x,y)dt \asymp
        \frac{\delta_D(x)^{\alpha/2}\delta_D(y)^{\alpha/2}}{\left( \delta_D(y)\vee |x-y| \right)^d}.
    \end{align}

    \bigskip
    \emph{Case 1}: $d>\alpha$.
     For any   $s\geq T$, the
    argument in \cite{CKS} gives
    \begin{align}\label{E:Green-small-time}
        \int_0^s p_D(t,x,y)dt \asymp
        \frac{1}{|x-y|^{d-\alpha}}
        \left( 1\wedge \frac{\delta_D(x)^{\alpha/2}\delta_D(y)^{\alpha/2}}{|x-y|^\alpha} \right)
    \end{align}
    and, moreover, the constants implicit in \eqref{E:Green-small-time} do not depend on $s>0$.
Taking $s\uparrow \infty$ gives the desired estimate for $G_D(x, y)$.

    \bigskip
    \emph{Case 2}: $d = 1 = \alpha$.  In this case \eqref{E:Green-small-time} doesn't hold, but we have the weaker estimate \eqref{E:Green-u-large}.  To see that this is weaker, compare with \eqref{E:Green-u} and recall that $|x-y|^\alpha/T \le 1$.  Define
    \begin{align}\label{E:u_0}
        u_0 := \frac{\delta_D(x)^{\alpha/2}\delta_D(y)^{\alpha/2}}{|x-y|^\alpha}.
    \end{align}
    As in \cite{CKS}, we have
    \begin{align*}
        &\frac{1}{|x-y|^{d-\alpha}}\int_{\frac{|x-y|^\alpha}{T}}^1
        \left( u^{(d/\alpha)-2}\wedge u^{-3} \right)
        \left( 1\wedge \frac{\sqrt u \delta_D(x)^{\alpha/2}}{|x-y|^{\alpha/2}} \right)
        \left( 1\wedge \frac{\sqrt u \delta_D(y)^{\alpha/2}}{|x-y|^{\alpha/2}} \right) du \notag \\
        \asymp\quad
        &\log\left( u_0\vee 1 \right) + u_0\left( \frac{1}{u_0}\wedge 1 - \frac{|x-y|^\alpha}{T} \right).
    \end{align*}
    By
     \eqref{E:Green-u}--\eqref{E:Green-t-large},
    and the previous display we have
    \begin{align*}
        G_D(x,y) &= \int_0^T p_D(t,x,y) dt + \int_T^\infty p_D(t,x,y) dt \\
        &\asymp \log(u_0\vee 1) + u_0\left( \frac{1}{u_0 }\wedge 1 - \frac{|x-y|^\alpha}{T} \right) + 1\wedge u_0 +
         \frac{\delta_D(x)^{\alpha/2}\delta_D(y)^{\alpha/2}}{(\delta_D(y)
        \vee |x-y|)^\alpha} \\
        &\asymp \log(u_0\vee 1) + 1\wedge u_0 \\
        &\asymp \log(1 + u_0).
    \end{align*}

    \bigskip
    \emph{Case 3}: $d = 1 < \alpha$.  Let $u_0$ be given by \eqref{E:u_0}.  Following \cite{CKS} we have
    \begin{align*}
        &\frac{1}{|x-y|^{d-\alpha}}\int_{\frac{|x-y|^\alpha}{T}}^1
        \left( u^{(d/\alpha)-2}\wedge u^{-3} \right)
        \left( 1\wedge \frac{\sqrt u \delta_D(x)^{\alpha/2}}{|x-y|^{\alpha/2}} \right)
        \left( 1\wedge \frac{\sqrt u \delta_D(y)^{\alpha/2}}{|x-y|^{\alpha/2}} \right) du \notag \\
        \asymp\quad
        & \frac{1}{|x-y|^{1-\alpha}}
        \left( \left( (u_0\vee 1)^{1 - (1/\alpha)} - 1 \right) + u_0\left( (u_0\vee 1)^{-1/\alpha} - \left( \frac{|x-y|^\alpha}{T} \right)^{1/\alpha} \right) \right).
    \end{align*}
    Hence by
     \eqref{E:Green-u}--\eqref{E:Green-t-large},
    and the previous display we have
    \begin{align*}
        &G_D(x,y) \\
        &= \int_T^\infty p_D(t,x,y)dt + \int_0^T p_D(t,x,y)dt \\
        &\asymp \frac{\delta_D(x)^{\alpha/2}\delta_D(y)^{\alpha/2}}{\delta_D(y)\vee |x-y|}
        + \frac{1}{|x-y|^{1-\alpha}}(1\wedge u_0) \\
        &\qquad
        + \frac{1}{|x-y|^{1-\alpha}}
        \left( \left( (u_0\vee 1)^{1 - (1/\alpha)} - 1 \right) + u_0\left( (u_0\vee 1)^{-1/\alpha} - \left( \frac{|x-y|^\alpha}{T} \right)^{1/\alpha} \right) \right) \\
        &\asymp
        \frac{\delta_D(x)^{\alpha/2}\delta_D(y)^{\alpha/2}}{\delta_D(y)\vee |x-y|}
        + \frac{1}{|x-y|^{1-\alpha}}\left( u_0\wedge u_0^{1-(1/\alpha)} \right) \\
        &=
        \frac{\delta_D(x)^{\alpha/2}\delta_D(y)^{\alpha/2}}{\delta_D(y)\vee |x-y|}
        + \frac{1}{|x-y|^{1-\alpha}}\left( \frac{\delta_D(x)^{\alpha/2}\delta_D(y)^{\alpha/2}}{|x-y|^\alpha} \wedge \frac{\delta_D(x)^{(\alpha-1)/2}\delta_D(y)^{(\alpha-1)/2}}{|x-y|^{\alpha-1}} \right) \\
        &\asymp
        \left( \delta_D(x)\delta_D(y) \right)^{(\alpha-1)/2}
        \wedge \frac{\delta_D(x)^{\alpha/2}\delta_D(y)^{\alpha/2}}{|x-y|}.
    \end{align*}

\end{proof}

\section{Heat kernel estimates on exterior open sets}
\label{sec:exterior}

We now turn to the proof of Theorem \ref{T:exterior}.  Assume
throughout this section that $d > \alpha$.  Recall that we say $D$
is an exterior open set if $D$ is a $C^{1,1}$ open set and $D^c$ is
compact.  Under the assumption $d>\alpha$, the process $X$ is
transient: for a compact set $K$ and a distant point $x$, with large
probability the process started at $x$ never visits $K$.  This is
the essence of the following key proposition in our proof of Theorem
\ref{T:exterior}.

\begin{prop}\label{P:exterior-compare}
    Let $d>\alpha$ and $D \subset \R^d$ be an exterior open set.  There exists an $R>0$
    with $D^c\subset B(0, R/2)$
    such that
    \begin{align}
        p_D(t,x,y) \asymp p(t,x,y)
        \label{E:exterior-compare}
    \end{align}
    on $(0,\infty)\times B(0,R)^c \times B(0,R)^c$.
\end{prop}

The idea of the proof for Theorem \ref{T:exterior} is nearly
identical to that of  Theorem \ref{T:half-space-like}, with
Proposition \ref{P:exterior-compare} playing the role of
\eqref{E:half-space-compare} in this case.  Note that once we have proved Theorem \ref{T:exterior}, we will have \eqref{E:exterior-compare} on $(0,\infty)\times D_r \times D_r$ for any $r>0$, where $D_r := \left\{ x \in D: \delta_D(x) > r \right\}$.

As in the previous
section, we need a way of ``pushing'' points in $D$ away from the
boundary.  For $x$ and $y$ in $D$, let $v \in \R^d$ be any unit
vector satisfying $\ip{x,v} \ge 0$ and $\ip{y,v} \ge 0$.  Let $R$ be
given by Proposition \ref{P:exterior-compare} and define
\begin{align}\label{E:outward-push}
    x_0 := x + Rv
    \quad\text{and}\quad
    y_0 := y + Rv.
\end{align}
By Pythagoras' Theorem,
\begin{align*}
    |x_0|^2 = (\ip{x,v} + R)^2 + |x - \ip{x,v}v|^2
    \ge R^2,
\end{align*}
and similarly, $|y_0| \ge R$.  Furthermore, applying Lemma \ref{L:ratio} with $t_0 = 1$ gives
\begin{align}
    \frac{p_D(1,x,z)}{p_D(1,x_0,z)} \asymp (1\wedge \delta_D(x)^\alpha)^{1/2}
    \quad\text{and}\quad
    \frac{p_D(1,y,z)}{p_D(1,y_0,z)} \asymp (1\wedge \delta_D(y)^\alpha)^{1/2}.
    \label{E:ratio-exterior}
\end{align}
Note that \eqref{E:ratio-exterior} does not depend on the particular choice of $v$ in \eqref{E:outward-push}.  We now prove Theorem \ref{T:exterior} assuming Proposition \ref{P:exterior-compare}.

\begin{proof}[Proof of Theorem \ref{T:exterior}]
    By \eqref{E:global-bound} and Theorem \ref{T:1.1}, it suffices to show that there exists a $T>0$ such that for all $t\ge T$ we have
    \begin{align*}
        p_D(t,x,y) \asymp (1\wedge \delta_D(x)^\alpha)^{1/2}(1\wedge \delta_D(y)^\alpha)^{1/2}p(t,x,y).
    \end{align*}
    Indeed, by the semigroup property, \eqref{E:ratio-exterior}, and Proposition \ref{P:exterior-compare} we have for all $t > 2$
    \begin{align*}
        p_D(t,x,y) &= \int_D\int_D p_D(1,x,z)p_D(t-2,z,w)p_D(1,w,y)dwdz \\
        &\asymp (1\wedge \delta_D(x)^\alpha)^{1/2}(1\wedge \delta_D(y)^\alpha)^{1/2} \\
        & \qquad\qquad\cdot\int_D\int_D p_D(1,x_0,z)p_D(t-2,z,w)p_D(1,w,y_0)dwdz \\
        &= (1\wedge \delta_D(x)^\alpha)^{1/2}(1\wedge \delta_D(y)^\alpha)^{1/2}p_D(t,x_0,y_0) \\
        &\asymp (1\wedge \delta_D(x)^\alpha)^{1/2}(1\wedge \delta_D(y)^\alpha)^{1/2}p(t,x_0,y_0).
    \end{align*}
    Finally, $p(t,x,y) = p(t,x_0,y_0)$ by translation invariance.
\end{proof}

The real work in this section is proving Proposition
\ref{P:exterior-compare}; this makes the case of an exterior open
set rather more involved than the case of a half-space-like one,
which admits the cheap estimate \eqref{E:half-space-compare}.   Such
a proposition was proved for Brownian motion by Grigor\'yan and
Saloff-Coste in \cite{GSC}, and, indeed, Lemma \ref{L:epsilon} is adapted from the proof of \cite[Theorem 3.3]{GSC}.  In \cite{GSC}, the authors use the technique of chaining together Harnack inequalities to prove the lower bound.  This, however, gives an exponential lower bound, which makes it inappropriate in the stable process case.  Instead, we use Dynkin's formula \eqref{E:dynkin's-formula} to obtain the correct lower bound (see the proof of Lemma \ref{L:exterior-lb-large-distance}).  In the remainder of this section, we will state and prove several lemmas, culminating in the proof of
Proposition \ref{P:exterior-compare}.  We begin by formulating the
transience of $X$ in a precise manner.  By compactness, we may fix
$R_0>0$ so that $D^c \subset B(0,R_0)$.  Recall that for a Borel set $B$, we write $T_B$ for the hitting time of $B$.

\begin{lemma}\label{L:C0} Let $B := B(0,R_0)$.  There is a constant $C_0 =C_0(d, \alpha, R_0)>0$ such that
$$
C_0^{-1} (1\wedge |x|^{\alpha-d}) \le \P_x (T_{B}<\infty) \leq C_0 (1\wedge |x|^{\alpha-d})
\qquad \hbox{for } |x|>R_0.
$$
\end{lemma}

\begin{proof}  Let $G(x, y)$ be the Green function of the symmetric $\alpha$-stable process on
$\R^d$. Then  for $x\in \R^d$,
$$ G\1_{B}(x)= \int_{B } G(x, y) dy = \int_{B } c |x-y|^{\alpha-d} dy
 \asymp 1\wedge  |x|^{\alpha -d}. $$
 By the strong Markov property of $X$, for $|x|>R_0$,
 $$ G\1_{B }(x) = \E_x \left[ G\1_{B }  (X_{T_{B }}); \,
   T_{B }<\infty \right] \asymp \P_x (T_{B }<\infty).
 $$
The conclusion of the lemma now follows from the above two displays.
\end{proof}

Our next lemma is a consequence of the parabolic Harnack inequality \cite[Proposition 4.3]{CK}.  This inequality applies to parabolic functions, which we now define.  For this we need to introduce the time-space process $Z_s := (V_s, X_s)$ where $V_s = V_0 + s$.  The law of the time-space process $s \mapsto Z_s$ starting from $(t,x)$ will be denoted by $\Pr_{(t,x)}$.  We say that a non-negative Borel measurable function $q(t,x)$ on $[0,\infty)\times \R^d$ is parabolic in a relatively open subset $\Omega$ of $[0,\infty)\times \R^d$ if for every relatively compact open subset $\Omega_1$ of $\Omega$, $q(t,x) = \E_{(t,x)}[q(Z_{\tau_{\Omega_1}})]$ for every $(t,x) \in \Omega_1$.  In particular, for a $C^{1,1}$ open set $D$, $y \in D$, and $T > T_0$, the function $q(t,x) := p_D(T-t,x,y)$ is parabolic on $[0,T_0]\times D$.

\begin{lemma}\label{L:phi}
    There exist constants $c > 0$ and $\gamma > 0$ such that
    for all $y \in D$ and  $t > 3\gamma \delta_D(y)^\alpha$, we have
        \begin{align*}
            \inf_{z\in B(y,\delta_D(y)/3)}p_D(t,z,y) \ge c p_D(t - \gamma \delta_D(x)^\alpha, y, y).
        \end{align*}
\end{lemma}

\begin{proof}
    Let $r:= \delta_D(y)$ and $D_y := r^{-1}(D-y)$.  By the parabolic Harnack inequality
    \cite[Proposition 4.3]{CK}, there exist constants $c>0, \gamma>0$
    such that for any non-negative function $q$ that is parabolic on $[0,3\gamma]\times B(0,1)$, we have
    \begin{align}
        \inf_{w\in B(0,1/3)}q(0,w) \ge cq(\gamma, 0).
        \label{E:phi-proof}
    \end{align}
    Let $t > 3\gamma r^\alpha$ and observe that the following function is parabolic on $[0,3\gamma]\times B(0,1)$:
    \begin{align*}
        q(s,w) := r^{-d}p_{D_y}(r^{-\alpha}t - s,w,0).
    \end{align*}
    By translation, scaling, and \eqref{E:phi-proof} we have
    \begin{align*}
        \inf_{z \in B(y,r/3)}p_D(t,z,y)
        &= \inf_{w \in B(0,1/3)}r^{-d}p_{D_y}(r^{-\alpha}t,w,0) \\
        &\ge cr^{-d}p_{D_y}(r^{-\alpha}t - \gamma,0,0) \\
        &= c p_D(t - \gamma r^\alpha, y, y).
    \end{align*}
\end{proof}

The next lemma is adapted from the proof of Theorem 3.3 in \cite{GSC}; it states that if $x$ is far from the boundary, and $y$ is a point near $x$, then with positive probability the process started from $x$ is near $y$ after $t$ units of time.

\begin{lemma}\label{L:epsilon}
    Let $A\ge 1$ be fixed.  There exist constants $\varepsilon = \varepsilon(d,\alpha,A) > 0$,
     $R_1 = R_1(d,\alpha,A) >R_0$ such that the following holds: for all $t>0$,
      $|x|>R_1$ and $y \in B(x,At^{1/\alpha})\cap D$,
    \begin{align*}
        \Pr_x\left( X_t^D \in B(y,t^{1/\alpha}) \right) \ge \varepsilon.
    \end{align*}
\end{lemma}

\begin{proof}
    We start by rewriting the probability:
    \begin{align}
        \Pr_x\left( X_t^D \in B(y,t^{1/\alpha}) \right)
        &= \Pr_x\left( \tau_D > t \right) -
        \Pr_x\left( X_t^D \not\in B(y,t^{1/\alpha}); \, \tau_D > t \right) \notag \\
        &\ge \Pr_x(\tau_D>t) - \Pr_x\left( X_t \not\in B(y,t^{1/\alpha}) \right).
        \label{eq:epsilon-1}
    \end{align}
    We will consider the two terms in \eqref{eq:epsilon-1} separately.  By scaling, translation, and the global estimate \eqref{E:global-bound} we have
    \begin{align*}
        \Pr_x\left( X_t \in B(y,t^{1/\alpha}) \right)
        &\ge \inf_{w \in B(0,A)} \Pr_w\left( X_1 \in B(0,1) \right) \\
        &\ge \inf_{w \in B(0,A)} C_1^{-1}\int_{B(0,1)}\left( 1\wedge \frac{1}{|w-z|^{d+\alpha}} \right)dz \\
        &\ge C_1^{-1}(A+1)^{-d-\alpha}|B(0,1)|.
    \end{align*}
    This last quantity is bounded away from 0, hence we may take $\varepsilon > 0$ small so that for $x \in \mathbb R^d$ and $y \in B(x,At^{1/\alpha})$ we have
    \begin{align}
        \Pr_x\left( X_t \not\in B(y,t^{1/\alpha}) \right)
        = 1 - \Pr_x\left( X_t \in B(y,t^{1/\alpha}) \right)
        \le 1-2\varepsilon.
        \label{eq:epsilon-2}
    \end{align}

    By assumption, $d>\alpha$ and we may choose $R_1>R_0$ so that $C_0R_1^{\alpha-d} \le \varepsilon$.  By Lemma \ref{L:C0}, for all $x$ with $|x| > R_1$ we have
    \begin{align*}
        \Pr_x\left( \tau_D \le t \right)
        \le \Pr_x\left( T_B < \infty \right)
        \le C_0\left( 1\wedge |x|^{\alpha-d} \right)
        \le C_0R_1^{\alpha-d}
        \le \varepsilon,
    \end{align*}
    hence,
    \begin{align}
        \Pr_x\left( \tau_D > t \right) = 1 - \Pr_x\left( \tau_D \le t \right) \ge 1 - \varepsilon.
        \label{eq:epsilon-3}
    \end{align}
    Finally, combining \eqref{eq:epsilon-1}, \eqref{eq:epsilon-2}, and \eqref{eq:epsilon-3} gives
    \begin{align*}
        \Pr_x\left( X_t^D \in B(y,t^{1/\alpha}) \right)
        &\ge \Pr_x\left( \tau_D>t \right) - \Pr_x\left( X_t \not\in B(y,t^{1/\alpha}) \right) \\
        &\ge \left( 1-\varepsilon \right) - \left( 1 - 2\varepsilon \right) \\
        &= \varepsilon.
    \end{align*}
\end{proof}

Following \cite{GSC}, we can use Lemma \ref{L:epsilon} to prove the following on-diagonal estimate for $p_D$: let $y \in D$ with $|y|>R_1$ and let $t > 0$.  By the semigroup property and the Cauchy-Schwarz inequality, we have
\begin{align*}
    p_D(2t,y,y)
    &\ge \int_{B(y,t^{1/\alpha})\cap D} \left[ p_D(t,y,z) \right]^2 dz \\
    &\ge \frac{1}{|B(y,t^{1/\alpha}) \cap D|}\Pr_y\left( X_t^D \in B(y,t^{1/\alpha}) \right)^2 \\
    &\ge \frac{\varepsilon^2}{|B(y,t^{1/\alpha})|}.
\end{align*}
Thus, there exists a $c>0$ such that
\begin{align}
    p_D(t,y,y) \ge ct^{-d/\alpha} \qquad \hbox{for } |y|>R_1 \hbox{ and } t>0.
    \label{eq:on-diagonal}
\end{align}

The next two lemmas divide the proof of Proposition \ref{P:exterior-compare} into two cases, depending on the relative sizes of $|x-y|$ and $t^{1/\alpha}$.

\begin{lemma}\label{L:exterior-lb-large-time}
    Let $R_1$ and $\varepsilon$ be given by Lemma \ref{L:epsilon} for $A=9$,
    and define $R_2 := \max\{R_1, 3R_0\}$.  There exist constants $\Lambda>0$, $C_2>0$
    such that for all $x,y \in D$, $|x|>R_2$, $|y|>R_2$, $|x-y| \le \Lambda t^{1/\alpha}$,
     and $t \ge (2R_0/\Lambda )^\alpha$, we have
    \begin{align*}
        p_D(t,x,y) \ge C_2t^{-d/\alpha}.
    \end{align*}
\end{lemma}

\begin{proof}
    Let $c_1$ and $\gamma$ be given by Lemma \ref{L:phi}, and fix $\Lambda $ sufficiently small so that
    \begin{align}
        (3^{-\alpha}+\gamma)3^\alpha \Lambda ^\alpha < \frac{1}{2}
        \quad\text{and}\quad
        (3^{-\alpha}+3\gamma)3^\alpha \Lambda ^\alpha < 1.
        \label{E:Lambda-def}
    \end{align}
    Assume without loss of generality that $|y| \ge |x|$ and define $y' := y/|y|$.  Define the half-space
    \begin{align*}
        H := \left\{ w: \ip{w, y'} > R_0 \right\}.
    \end{align*}
    That is, $H \subset D$ is the half-space tangent to $B(0,R_0)$ with inward unit normal $y'$.

    The proof is divided into two cases.

    \bigskip
    \emph{Case 1}: $\delta_D(y) > 3\Lambda t^{1/\alpha}$.  Note that
    \begin{align*}
        \delta_H(x) \le \delta_{B(0,R_0)}(x) \le \delta_{B(0,R_0)}(y) = \delta_H(y),
    \end{align*}
    where the second inequality follows from the assumption $|x| \le |y|$.  In addition, $\delta_H(y) \ge \delta_D(y) - 2R_0$ and $t > (2R_0/\Lambda)^\alpha$ imply
    \begin{align*}
        \delta_H(y) \ge \delta_H(x) \ge \delta_H(y) - \Lambda t^{1/\alpha} \ge 2\Lambda t^{1/\alpha} - 2R_0 \ge \Lambda t^{1/\alpha}.
    \end{align*}
    By Lemma \ref{L:half-space} there exists a constant $c_2>0$ such that
    \begin{align*}
        p_D(t,x,y)
        &\ge p_H(t,x,y) \\
        &\ge c_2\left(  1\wedge \frac{\delta_H(x)^\alpha}{t}\right)^{1/2}
        \left(  1\wedge \frac{\delta_H(y)^\alpha}{t}\right)^{1/2}
        \left( 1\wedge \frac{t^{1/\alpha}}{|x-y|} \right)^{d+\alpha}t^{-d/\alpha} \\
        &\ge c_2\left( 1\wedge \Lambda ^\alpha \right)^{1/2} \left( 1\wedge \Lambda ^\alpha \right)^{1/2}\left( 1\wedge \frac{1}{\Lambda } \right)^{d+\alpha}t^{-d/\alpha} \\
        &= c_3t^{-d/\alpha}.
    \end{align*}

    \bigskip
    \emph{Case 2}: $\delta_D(y) \le 3\Lambda t^{1/\alpha}$.  Define $t_0 := \left( \delta_D(y)/3 \right)^\alpha$.  Observe that by \eqref{E:Lambda-def} we have
    \begin{align*}
        t_0 + 3\gamma \delta_D(y)^\alpha
        = (3^{-\alpha}+3\gamma)\delta_D(y)^\alpha
        \le (3^{-\alpha} + 3\gamma)3^\alpha \Lambda ^\alpha t
        < t.
    \end{align*}
    Hence $t-t_0 > 3\gamma \delta_D(y)^\alpha$.  By the semigroup property and Lemma \ref{L:phi} we have
    \begin{align}
        p_D(t,x,y)
        &\ge \int_{B(y,t_0^{1/\alpha})}p_D\left( t_0, x, z \right)p_D\left( t - t_0,z,y \right)dz \notag \\
        &\ge \Pr_x\left( X_{t_0}^D \in B(y,t_0^{1/\alpha}) \right)c_1p_D(t-t_0-\gamma \delta_D(y)^\alpha,y,y).
        \label{E:exterior-lb}
    \end{align}
    We want to use Lemma \ref{L:epsilon} to bound the first term in \eqref{E:exterior-lb}.  Observe that
    \begin{align*}
        |x-y| \le 2|y| \le 2(\delta_D(y) + R_0) < 3\delta_D(y),
    \end{align*}
    where the last inequality follows because $|y| > R_2 \ge 3R_0$ implies $\delta_D(y) > 2R_0$.  Thus, $y \in B(x,9t_0^{1/\alpha})$ and Lemma \ref{L:epsilon} gives $\Pr_x(X_{t_0}^D \in B(y,t_0^{1/\alpha})) \ge \varepsilon$.

    To bound the second term in \eqref{E:exterior-lb} we note that by \eqref{E:Lambda-def} we have
    \begin{align*}
        t - t_0 - \gamma \delta_D(y)^\alpha
        &= t - (3^{-\alpha} + \gamma)\delta_D(y)^\alpha \\
        &\ge t - (3^{-\alpha} + \gamma)3^\alpha \Lambda ^\alpha t \\
        &> \frac{1}{2}t.
    \end{align*}
    Hence $t - t_0 - \gamma \delta_D(y)^\alpha \asymp t$, and by \eqref{eq:on-diagonal} there is a constant $c_4>0$ for which $p_D(t - t_0 - \gamma \delta_D(y^\alpha),y,y) \ge c_4t^{-d/\alpha}$.
Thus \eqref{E:exterior-lb} can be continued
    \begin{align*}
        p_D(t,x,y) \ge \varepsilon c_1 c_4 t^{-d/\alpha}
        = c_5t^{-d/\alpha}.
    \end{align*}
    The lemma is now proved with $C_2 := \min\{c_3, c_5\}$.
\end{proof}

For our last lemma, we let $R_3 > R_0$ be chosen to satisfy
\begin{equation}\label{e:R1}
    C_1^2 \Lambda^{-d-\alpha}4^{d+\alpha} C_0 R_3^{\alpha-d}<1/2,
\end{equation}
where $C_1$ is given by \eqref{E:global-bound}, $\Lambda$ is given by Lemma \ref{L:exterior-lb-large-time}, and $C_0$ is the constant in Lemma \ref{L:C0}.

\begin{lemma}\label{L:exterior-lb-large-distance}
    There is a constant $C_3=C_3(d, \alpha, R_0)>0$ such that for every $t\geq  (4R_3/\Lambda)^{\alpha}$ and $x, y\in D$ with $|x|>R_3$, $|y|>R_3$, and $|x-y|> \Lambda t^{1/\alpha}$,
$$
p_D(t, x, y) \geq C_3 \frac{t}{|x-y|^{d+\alpha}}.
$$
\end{lemma}

\begin{proof}
    One of $|x|$ and $|y|$ should be larger than $|x-y|/2$; we assume without loss of generality that $|y|\ge|x-y|/2$.  Since $|x-y| > \Lambda t^{1/\alpha} > 4R_0$, we get
    \begin{align*}
        |z-y| > |x-y|/4 \qquad \hbox{for every } z\in B(0, R_0).
    \end{align*}
    Let $B:=\{z\in \R^d: \, |z|\leq R_0\}$ and $T:=T_B$.  By Dynkin's formula \eqref{E:dynkin's-formula}, Lemma \ref{L:C0}, and \eqref{e:R1}
\begin{eqnarray*}
    p_D(t, x, y) &\geq& p_{B^c} (t, x, y) \\
    &=& p(t, x, y) -\E_x \left[ p(t-T, X_T, y); T<t\right]\\
    &\geq & C_1^{-1}\Lambda^{d+\alpha} \frac {t}{|x-y|^{d+\alpha}} - C_1 \E_x \left[ \frac{t-T}{|X_T-y|^{d+\alpha}}; T<t\right] \\
    &\geq & C_1^{-1}\Lambda^{d+\alpha} \frac {t}{|x-y|^{d+\alpha}} - C_1 \frac{t}{(|x-y|/4)^{d+\alpha}}\P_x (T<t)\\
    &\geq & C_1^{-1}\Lambda^{d+\alpha} \frac {t}{|x-y|^{d+\alpha}} \left( 1- C_1^2 \Lambda^{-d-\alpha}4^{d+\alpha}\P_x (T<\infty)\right) \\
    &> & \frac{\Lambda^{d+\alpha}}{2C_1} \frac {t}{|x-y|^{d+\alpha}} .
\end{eqnarray*}
This proves the theorem with $C_3 := \Lambda^{d+\alpha}/2C_1$.
\end{proof}

We can now give the proofs of Proposition \ref{P:exterior-compare} and Corollary \ref{C:exterior}.

\begin{proof}[Proof of Proposition \ref{P:exterior-compare}]
    We have $p_D(t,x,y) \le p(t,x,y)$ trivially, hence we need only show the lower bound in \eqref{E:exterior-compare}.  Fix $R := \max\{2, R_2, R_3\}$, $T := (4R_3/\Lambda)^\alpha$, and suppose $|x| > R$ and $|y| > R$.  For $t \le T$, the lower bound follows by Theorem \ref{T:1.1} and \eqref{E:global-bound}, and for $t > T$ it's a consequence of Lemmas \ref{L:exterior-lb-large-time} and \ref{L:exterior-lb-large-distance} and \eqref{E:global-bound}.
\end{proof}

\begin{proof}[Proof of Corollary \ref{C:exterior}]
    We assume throughout, without loss of generality, that $\delta_D(x) \le \delta_D(y)$.  Let $r := 1 + 2^{1/2\alpha}$.  We divide the proof into three cases, proving the following estimates:
    \begin{align*}
        G_D(x,y) &\asymp
        \frac{1}{|x-y|^{d-\alpha}}
        \left( 1\wedge \frac{\delta_D(x)^{\alpha/2}}{|x-y|^{\alpha/2}} \right)
        \left( 1\wedge \frac{\delta_D(y)^{\alpha/2}}{|x-y|^{\alpha/2}} \right)
        & \text{when $\delta_D(x) \le \delta_D(y) \le r$,} \\
        G_D(x,y) &\asymp
        \frac{1}{|x-y|^{d-\alpha}}
        & \text{when $1 < \delta_D(x) \le \delta_D(y)$, and} \\
        G_D(x,y) &\asymp
        \frac{1}{|x-y|^{d-\alpha}}
        \delta_D(x)^{\alpha/2}
        & \text{when $\delta_D(x) \le 1 < r < \delta_D(y)$.}
    \end{align*}
    This is equivalent to the conclusion in Corollary \ref{C:exterior}.

    \bigskip
    \emph{Case 1}:  $\delta_D(x) \le r$ and $\delta_D(y) \le r$.  Note that under this assumption, $|x-y| \le 2r+2R_0$.  If we define $T = (2r+2R_0)^\alpha$, then the proof of \cite[Corollary 1.2]{CKS} (for the case $d>\alpha$) applies to give
    \begin{align*}
        G_D(x,y) \asymp \frac{1}{|x-y|^{d-\alpha}}
        \left( 1\wedge \frac{\delta_D(x)^{\alpha/2}}{|x-y|^{\alpha/2}} \right)
        \left( 1\wedge \frac{\delta_D(y)^{\alpha/2}}{|x-y|^{\alpha/2}} \right).
    \end{align*}

    \bigskip
    \emph{Case 2}:  $\delta_D(x) > 1$ and $\delta_D(y) > 1$.  Then by Theorem \ref{T:exterior},
    \begin{align*}
        G_D(x,y) \asymp
        \int_0^\infty \left( t^{-d/\alpha}\wedge \frac{t}{|x-y|^{d+\alpha}} \right)
        = c_1 \frac{1}{|x-y|^{d-\alpha}}.
    \end{align*}

    \bigskip
    \emph{Case 3}: $\delta_D(x) \le 1 < r < \delta_D(y)$.  Observe that in this case, $|x-y| \ge \delta_D(y) - \delta_D(x) > 2^{1/2\alpha}$.  We have
    \begin{align}\label{E:C2-1}
        \int_1^\infty p_D(t,x,y)dt
        \asymp \delta_D(x)^{\alpha/2}
        \int_1^\infty \left( t^{-d/\alpha}\wedge \frac{t}{|x-y|^{d+\alpha}} \right)dt.
    \end{align}
    and
    \begin{align}\label{E:C2-2}
        \int_1^\infty \left( t^{-d/\alpha}\wedge \frac{t}{|x-y|^{d+\alpha}} \right)dt
        &= \int_1^{|x-y|^\alpha} \frac{t}{|x-y|^{d+\alpha}}dt
        + \int_{|x-y|^\alpha}^\infty t^{-d/\alpha}dt \notag \\
        &= \frac{1}{2}\frac{1}{|x-y|^{d+\alpha}}\left( |x-y|^{2\alpha} - 1 \right)
        + \frac{\alpha}{d-\alpha}\frac{1}{|x-y|^{d-\alpha}}.
    \end{align}
    In addition, $|x-y| > 2^{1/2\alpha}$ implies $|x-y|^{2\alpha} - 1 > \frac{1}{2}|x-y|^{2\alpha}$.  Hence \eqref{E:C2-1} and \eqref{E:C2-2} give
    \begin{align}\label{E:C2-3}
        \int_1^\infty p_D(t,x,y) \asymp
        \frac{\delta_D(x)^{\alpha/2}}{|x-y|^{d-\alpha}}.
    \end{align}
    On the other hand,
    \begin{align}\label{E:C2-4}
        \int_0^1 p_D(t,x,y)dt &\asymp
        \int_0^1 \left( 1 \wedge \frac{\delta_D(x)^{\alpha/2}}{1\wedge \sqrt t} \right)\left( t^{-d/\alpha}\wedge \frac{t}{|x-y|^{d+\alpha}} \right)dt \notag \\
        &= \int_0^{\delta_D(x)^\alpha} \frac{t}{|x-y|^{d+\alpha}}dt
        + \int_{\delta_D(x)^\alpha}^1 \frac{\delta_D(x)^{\alpha/2}}{\sqrt t}\frac{t}{|x-y|^{d+\alpha}}dt \notag \\
        &= \frac{1}{|x-y|^{d+\alpha}}\left( \frac{2}{3}\delta_D(x)^{\alpha/2} - \frac{1}{6}\delta_D(x)^{2\alpha} \right) \notag \\
        &\asymp \frac{\delta_D(x)^{\alpha/2}}{|x-y|^{d+\alpha}}.
    \end{align}
    In the last estimate we used that fact that $\delta_D(x) \le 1$ implies $\delta_D(x)^{2\alpha} \le \delta_D(x)^{\alpha/2}$.  Combining \eqref{E:C2-3} and \eqref{E:C2-4} gives
    \begin{align*}
        \int_0^\infty p_D(t,x,y)dt &\asymp
        \frac{1}{|x-y|^{d-\alpha}}\left( \frac{\delta_D(x)^{\alpha/2}}{|x-y|^{2\alpha}} + \delta_D(x)^{\alpha/2} \right)
        \asymp \frac{\delta_D(x)^{\alpha/2}}{|x-y|^{d-\alpha}}.
    \end{align*}
    In the last estimate, we use the fact again that $|x-y|^{2\alpha} > 2$.
\end{proof}

\section{Censored $\alpha$-stable process in exterior open sets}
\label{sec:c-exterior}

We begin this section by recalling the basic theory of the censored $\alpha$-stable process (see \cite{BBC} and \cite{CKS2} for a more detailed study).  Fix a $C^{1,1}$ open set $D$ in $\R^d$ with $d \ge 1$.  Define a bilinear form $\mathcal E$ on $C_c^\infty(D)$ by
\begin{align*}
    \mathcal E(u,v) := c
    \int_D \int_D (u(x)-u(y))(v(x)-v(y))\frac{dxdy}{|x-y|^{d+\alpha}},
    \quad u,v \in C_c^\infty(D),
\end{align*}
where $c = c(d,\alpha)$ is an appropriately chosen scaling constant.  Using Fatou's lemma, it is easy to check that the bilinear form $(\mathcal E, C_c^\infty(D))$ is closable in $L^2(D) = L^2(D,dx)$.  Let $\mathcal F$ be the closure of $C_c^\infty(D)$ under the Hilbert inner product $\mathcal E_1 := \mathcal E + (\cdot,\cdot)_{L^2(D)}$.  As noted in \cite{BBC}, $(\mathcal E, \mathcal F)$ is Markovian and hence a regular symmetric Dirichlet form on $L^2(D,dx)$, and therefore there is an associated symmetric Hunt process $Y = \left\{ Y_t, t\ge 0, \Pr_x, x \in D \right\}$ taking values in $D$.  The process $Y$ is called a censored $\alpha$-stable process in $D$.

Closely related to the censored process in $D$ is the reflected process in $\bar{D}$.  Define
\begin{align*}
    \mathcal F^{\text{ref}} :=
    \left\{ u \in L^2(D): \int_D\int_D \frac{(u(x)-u(y))^2}{|x-y|^{d+\alpha}} < \infty \right\}
\end{align*}
and
\begin{align*}
    \mathcal E^{\text{ref}}(u,v) :=
    c\int_D\int_D (u(x)-u(y))(v(x)-v(y))\frac{dxdy}{|x-y|^{d+\alpha}},
    \quad u,v \in \mathcal F^{\text{ref}}.
\end{align*}
It is shown in \cite[Remark 2.1]{BBC} that the bilinear form $(\mathcal E^{\text{ref}},\mathcal F^{\text{ref}})$ is a regular symmetric Dirichlet form on $L^2(\bar{D})$.  The process $\bar{Y}$ on $\bar{D}$ associated with $(\mathcal E^{\text{ref}},\mathcal F^{\text{ref}})$ is called a reflected $\alpha$-stable process on $\bar{D}$.  In some sense, $\bar{Y}$ represents a maximal extension of $Y$, and the censored $\alpha$-stable process $Y$ can be realized as a subprocess of $\bar{Y}$ killed upon exiting $D$ (see \cite[Remark 2.1]{BBC}).

If we let $q_D(t,x,y)$ and $\bar{q}_D(t,x,y)$ denote the transition densities of $Y$ and $\bar{Y}$, respectively, then this last fact implies $q_D(t,x,y) \le \bar{q}_D(t,x,y)$
 on $(0, \infty)\times D \times D$.
In parts of our proof, we use this observation in place of domain monotonicity.
 Observe that an
 exterior $C^{1,1}$ open set is a so-called
 global $d$-set; that is there is a constant $c>1$ such that
 \begin{align} \label{E:d-set}
     c^{-1} r^d \leq | D\cap B(x, r)|\leq c \, r^d
     \qquad \hbox{for every } x\in D \hbox{ and } r>0.
 \end{align}
  Hence by \cite[Theorem 1.1]{CK} we have
\begin{align}\label{E:reflected-bound}
    \bar{q}_D(t,x,y) \asymp t^{-d/\alpha}\wedge \frac{t}{|x-y|^{d+\alpha}}
    \qquad \text{on } (0,\infty)\times \bar D \times \bar D .
\end{align}
That \eqref{E:reflected-bound} holds for all times $t > 0$ and not just until some finite time is due to the following facts.  First, for any $\lambda>0$, $\lambda^{-1}\bar{Y}_{\lambda^\alpha t}$ is a reflected $\alpha$-stable process in $\lambda^{-1}D$, and so
$$\bar{q}_D(t,x,y) = \lambda^{-d}\, \bar{q}_{\lambda^{-1}D}(\lambda^{-\alpha}t,\lambda^{-1}x,\lambda^{-1}y)
\qquad \hbox{for } t>0 \hbox{ and } x, y\in \bar D.
  $$
  Second, if $D$ is an open $d$-set with constant $c$ (i.e., $D$ satisfies \eqref{E:d-set}), then $\lambda^{-1}D$ is an open $d$-set with the same constant $c$.
Hence by \cite[Theorem 1.1]{CK}, there is $c_1>1$ so that for every $\lambda >0$,
 $$
 c_1^{-1} \left( t^{-d/\alpha} \wedge \frac{t}{|w-z|^{d+\alpha}}\right)
  \leq \bar{q}_{\lambda^{-1}D}(t, w, z) \leq c_1 \left( t^{-d/\alpha}
   \wedge \frac{t}{|w-z|^{d+\alpha}} \right)
  $$
holds for every $ t\in (0, 1]$ and  $w, z\in \bar{\lambda^{-1}D}$.
The last two displays yield global estimate \eqref{E:reflected-bound}
for $\bar q_D(t, x, y)$.

  When
 $D$ is a globally Lipschitz open set and $\alpha \in (0,1]$, it is proved in \cite{BBC} that $Y = \bar{Y}$ as processes.  In this case, \eqref{E:reflected-bound} already gives a sharp two-sided estimate for $q_D(t,x,y)$.  For this reason, we focus on the case $\alpha \in (1,2)$.

By \cite[Theorem 2.1]{BBC}, the censored $\alpha$-stable process $Y$ can be obtained from a countable number of copies of $X^D$ by the Ikeda-Nagasawa-Watanabe piecing together procedure.  In particular, we may assume the processes $Y$ and $X^D$ are coupled as follows:
\begin{equation}\label{e:4.3}
    Y_t = X_t^D = X_t \quad \text{for all $t < \tau^0_D$},
\end{equation}
where $\tau^0_D := \inf\left\{ t>0: X_t \not\in D
\right\}$.   Note that this gives the easy estimate
 \begin{equation}\label{e:4.4}
 p_D(t,x,y) \le q_D(t,x,y)  \qquad \hbox{on } (0, \infty)\times D \times D.
\end{equation}

We now turn to the proof of Theorem \ref{T:c-exterior}.  Assume for the remainder of the section that $\alpha \in (1,2)$ and $D$ is an exterior $C^{1,1}$ open set, that is, $D$ is a $C^{1,1}$ open set with compact complement.  As in Section \ref{sec:exterior}, we assume that $R_0>0$ is sufficiently large so that $D^c \subset B(0,R_0)$.  The steps in the proof are the same as for the proof of Theorem \ref{T:exterior}; our approach is to indicate how the statements and proofs of the previous section can be adapted to prove Theorem \ref{T:c-exterior}, frequently leaving the details to the reader.  For our first lemma, we state the analog of Lemma \ref{L:ratio}.  The proof of Lemma \ref{L:c-ratio} is the same as the proof of Lemma \ref{L:ratio} except Theorem \ref{T:c-1.1} takes the place of Theorem \ref{T:1.1}.
\begin{lemma}\label{L:c-ratio}
    Let $D$ be a $C^{1,1}$ open set, and let $\lambda >0$ and $t_0 >0$ be fixed.  Suppose $x,x_0 \in D$ satisfy $|x-x_0| = \lambda t_0^{1/\alpha}$.  Then
    \begin{align}
        \frac{q_D(t_0,x,z)}{q_D(t_0,x_0,z)} \asymp
        \frac{1\wedge \delta_D(x)^{\alpha-1}}{1\wedge \delta_D(x_0)^{\alpha-1}}
        \label{E:c-ratio}
    \end{align}
    as a function of $(x,x_0)$, uniformly in $z$.
    The implicit constant in \eqref{E:ratio} depends on $d$, $\alpha$, $\lambda$, $t_0$
     and the $C^{1,1}$ characteristics of $D$.
\end{lemma}
Thus, if we define $x_0$ and $y_0$ by \eqref{E:outward-push} then
\begin{align}
    \frac{q_D(1,x,z)}{q_D(1,x_0,z)} \asymp 1\wedge \delta_D(x)^{\alpha-1}
    \quad\text{and}\quad
    \frac{q_D(1,y,z)}{q_D(1,y_0,z)} \asymp 1\wedge \delta_D(y)^{\alpha-1}.
    \label{E:c-ratio-exterior}
\end{align}

As in the previous sections, our key proposition is an estimate on the interior of $D$.
\begin{prop}\label{P:c-exterior-compare}
    Let $d\ge 2 >\alpha$ and $D \subset \R^d$ be an exterior open set.  There exists an $R>0$
    with $D^c\subset B(0, R/2)$
    such that
    \begin{align}\label{E:c-exterior-compare}
        q_D(t,x,y) \asymp \bar{q}_D(t,x,y)
    \end{align}
    on $(0,\infty)\times B(0,R)^c \times B(0,R)^c$.
\end{prop}

If we assume Proposition \ref{P:c-exterior-compare}, then Theorem \ref{T:c-exterior} follows from the proof of Theorem \ref{T:exterior}, with \eqref{E:c-ratio-exterior} and Proposition \ref{P:c-exterior-compare} taking the places of \eqref{E:ratio-exterior} and Proposition \ref{P:exterior-compare}.

We will now show Proposition \ref{P:c-exterior-compare} by arguing
that  appropriate analogs to Lemmas \ref{L:C0} through
 \ref{L:exterior-lb-large-distance}
hold in the censored $\alpha$-stable process case.
 The following is an analogy of Lemma \ref{L:C0} in the context of
 censored stable process.

\begin{lemma}\label{L:c-C0} Denote by $B$ the closed ball centered at the origin
with radius $R_0$. Then there is a constant $c = c(d, \alpha,
R_0)>1$ such that
$$
 c^{-1} (1\wedge |x|^{\alpha-d}) \leq \P_x (T_{B}<\infty) \leq c (1\wedge |x|^{\alpha-d})
\qquad \hbox{for } |x|>R_0.
$$
\end{lemma}

\begin{proof}
    Let $\bar G_D(x, y)$ be the Green function of the reflected $\alpha$-stable process $\bar Y$ on $D$; that is,
     $\bar G_D(x, y)=\int_0^\infty \bar q_D(t, x, y) dt$.
    It follows from \eqref{E:reflected-bound}  that on $\overline D \times
    \overline D$,
    \begin{equation}
    \bar G_D(x, y) \asymp t^{-d/\alpha} \wedge \frac{t}{|x-y|^{d-\alpha}}.
    \end{equation}
     As the censored $\alpha$-stable process $Y$ is the part process
     of $\bar Y$ killed upon hitting $\partial D$,
     Lemma \ref{L:c-C0} follows from the proof of Lemma \ref{L:C0}.
\end{proof}

With Lemma \ref{L:c-C0}, the proof of
Lemma \ref{L:exterior-lb-large-distance} carries over for $Y$.
  Next, we observe that Lemma \ref{L:phi} applies to
$q_D(t,x,y)$ because the proof relies  only on the parabolic Harnack
inequality, which is proved for $Y$ in \cite[Proposition 4.3]{CK}.
Lemmas \ref{L:epsilon}  and \ref{L:exterior-lb-large-time} holds for
the censored stable process $Y$ because \eqref{e:4.4}.

\medskip

Lastly, the proof of Corollary \ref{C:exterior} can be easily adapted to
prove Corollary \ref{C:c-exterior}

\vskip 0.3truein

Department of Mathematics, University of Washington,
 Seattle, WA 98195, USA

Email: \texttt{zchen@math.washington.edu}  \ and \   \texttt{jtokle@math.washington.edu}

\end{document}